\documentclass[12pt]{amsart}

\usepackage{amssymb,amsfonts,amsthm}
\usepackage{amsmath}
\usepackage[all]{xy}

\usepackage[margin=3cm]{geometry}

\numberwithin{equation}{section}
\newtheorem{thm}[equation]{Theorem}
\newtheorem{lem}[equation]{Lemma}
\newtheorem{cor}[equation]{Corollary}
\newtheorem{prop}[equation]{Proposition}

\newtheorem{prob}[equation]{Problem}

\theoremstyle{definition}
\newtheorem{dfn}[equation]{Definition}
\newtheorem{exam}[equation]{Example}

\theoremstyle{remark}

\renewcommand{\leq}{\leqslant}
\renewcommand{\geq}{\geqslant}

\newcommand{\abs}[1]{\left\lvert#1\right\rvert}

\newcommand{\pigwedge}{\mathop{\bigwedge}\nolimits}

\newcommand{\Supp}{\mathop{\mathrm{Supp}}\nolimits}
\newcommand{\vol}{\mathop{\mathrm{vol}}\nolimits}

\begin{document}

\title[]{Remarks on $L^2$-jet extension and extension of 
singular Hermitian metric with semipositive curvature}
\author[]{Tomoyuki Hisamoto}
\address{Graduate School of Mathematical Sciences, The University of Tokyo, 
3-8-1 Komaba Meguro-ku, Tokyo 153-0041, Japan}
\email{hisamoto@ms.u-tokyo.ac.jp}
\subjclass[2000]{Primary~32U05, Secondary~32C25, 32L10, 32A25}
\keywords{singular Hermitian metric, extension theorem, Bergman kernel}
\date{}
\maketitle

\begin{abstract}

We give a new variant of $L^2$-extension theorem for the jets of holomorphic sections and discuss the relation between the extension problem of singular Hermitian metrics with semipositive curvature. 

\end{abstract}

\section{Introduction}\label{introduction}

In this paper we first give the following new variant of $L^2$-extension theorem for the jets of holomorphic sections. 

\begin{thm}\label{extension} 

  Let $X$ be a smooth projective variety with a fixed K\"{a}hler form $\omega$, and $S \subseteq X$ a smooth closed subvariety. Then there exist constants $N=N(S,  X, \omega)$ such that the following holds; 
Let $L \to X$ be a holomorphic line bundle with a {\em smooth} Hermitian metric $h$ whose curvature current $\Theta_h$ satisfies
   \begin{equation*}
      \Theta_h \geq N\omega. 
   \end{equation*}   
Then for any $m \geq 1$ and any section $f \in H^0(S, L^{\otimes m}) $ with 
   \begin{equation*}
     \int_S \abs{f}_{h^m}dV_{\omega, S}  < + \infty, 
   \end{equation*}
  there exists a section $ F \in H^0(X, L^{\otimes m}) $
  such that $F|_S=f$ and the every other term in $(m-1)$-jet along $S$ vanishes. That is, $J^{m-1}F|_S=f$ holds if one denotes the $(m-1)$-jet of $F$ along $S$ by $J^{m-1}F|_S$. Moreover, there exists a constant $C=C(S, X, h)>0$ such that the $L^2$-estimate 
   \begin{equation*}
      \int_X \abs{F}^2_{h^m}dV_{\omega, X} 
      \leq C^{m^2} \int_S  \abs{f}^2_{h^m}dV_{\omega, S}
   \end{equation*} 
  holds. 
\end{thm} 

Theorem \ref{extension} originates from \cite{OT87}. Compact manifold case was first treated by \cite{Man93} and Ohsawa considered arbitrary closed submanifolds in \cite{Ohs01}. For the jet-extension, there is \cite{Pop05} which generalizes Manivel's result assuming that the subvariety $S$ is defined as a locus of some holomorphic section of a vector bundle. Theorem \ref{extension} not only generalizes \cite{Pop05} to general submanifolds but also specifies how the coefficient in the $L^2$-estimate varies when one twists the line bundle. This is our new viewpoint. On the other hand, in the present paper we only treat with a very special type of jet (every term of $J^{m-1}F$ except the zeroth order term vanishes). 

The formulation of Theorem \ref{extension} is strongly motivated by the extension problem for singular Hermitian metrics with semipositive curvature. Let $L$ be a holomorphic line bundle over a smooth complex projective variety $X$. We fix a closed subvariety $S$ of $X$ and a singular Hermitian metric $h_S$ on $L|_{S}$. We will suggest a new approach to generalize the following result due to Coman, Guedj, and Zeriahi. 

\begin{thm}[smooth subvariety case of Theorem B in \cite{CGZ10}]\label{CGZ}
Let $X$ be a smooth complex projective variety, $L$ an ample line bundle over $X$, and $S\subset X$ a smooth closed subvariety. Then for any singular Hermitian metric $h_S$ with semipositive curvature on $L|_S$ there exists a singular Hermitian metric $h$ on $L$ over $X$ such that $h|_S = h_S$ holds. 
\end{thm} 

Note that \cite{CGZ10} treated arbitrary singular subvariety. A finite family of global sections $\{F_i\}_i \subset H^0(X, L)$ naturally defines a singular Hermitian metric $1/\sum_i \abs{F_i}^{2}$ on $L$ so that Proposition \ref{metric} can be seen as an analytic generalization of the Serre vanishing theorem for ample line bundles. In the local setting there were previous works by \cite{Ric68}, \cite{Sad82} and \cite{Col91}.  Following the approach of \cite{Col91}, \cite{CGZ10} obtained the above result as a consequence of the growth-control extension of plurisubharmonic functions from a closed subvariety in the complex Euclidian space. Their proof is hence rather pluripotential-theoritic. In this paper we study a new approach via $L^2$-extension theorem for holomorphic sections of the line bundle, to give a direct relation between the extendability of sections and that of metrics. Moreover, it enables us to expect a consistent proof in the general big line bundle case. We suggest the following problem. 

\begin{prob}\label{problem}
Let $X$ be a smooth complex projective variety, $L$ a big line bundle over $X$, and $S\subset X$ a smooth closed subvariety. Let us take a singular Hermitian metric $h_S$ with semipositive curvature on $L|_S$, and a singular Hermitian metric $h_0$ with strictly positive curvature  on $L$, so that $h_S \geq h_0|_S$ holds. Then does there exist a singular Hermitian metric $h$ on $L$ over $X$ such that $h|_S = h_S$ holds?. 
\end{prob} 

When $h_0$ is smooth the above problem is obviously reduced to Theorem \ref{CGZ}. For the non-ample line bundle, is not always possible to extend arbitrary singular metrics from $S$. We will discuss this point in subsection $4.3$. We show that Problem \ref{problem} comes down to the further refinement of Theorem \ref{extension}. 

\begin{prop}\label{metric}
If one can replace the $L^2$-coefficient $C^{m^2}$ in Theorem \ref{extension} to $C^{m}$ $(m\geq1)$, then we obtain the $L^2$-theoritic proof of Theorem \ref{CGZ}. 
If further $h$ in Theorem \ref{extension} can be taken singular, the same line solves Problem \ref{problem} affirmatively. 
\end{prop}

We remark that in the proof of Theorem \ref{extension} the constant $C$ in fact depends on the modulus of continuity of $h$ so that the smoothness assumption of $h$ can not be removed so far. 

Let us briefly explain the proof of Proposition \ref{metric}. First of all, one can approximate the metric $h_S$ on $S$ by a sequence of singular metric $h_{S, m}$ defined by holomorphic sections of $L^{\otimes m}(m\geq 1)$. Then we can apply $L^2$-extension theorem to extend each section to $X$ so that they produce a sequence of algebraic singular metric  $h_m$ on $L$ over $X$, which approximates $h_S$ on $S$. Further, one can get a convergent subsequence thanks to the $L^2$-estimate. The limit $h$ is naively to be desired extension of $h_S$. However, we unfortunately have $h|_S \neq h _S$ in general. To make $h_S$ dominate $h_{S, m}$ around $S$, one have to control the jet of the extended holomorphic sections and here we need the formulation of Theorem \ref{extension}.

\section{Jet extension}\label{jet extension}

Let us first fix notations. Let $E$ be a holomorphic Hermitian line bundle on a smooth projective variety $X$. 
If a K\"{a}hler metric $\omega$ is fixed, the Chern curvature $\Theta(E)$ defines a Hermitian form on $(\pigwedge^{p.q} T_{X,x}^*) \otimes E_x$ 
 as follows: 
 \begin{equation*}
  \theta ( \alpha , \beta ):= 
    \big( [\Theta(E),\Lambda]\alpha|\beta \big) 
    \ \text{ for } \ \alpha , \beta \in (\pigwedge^{p,q}T_{X,x}^*) \otimes E_x 
    \ \ \ \ \ ( x \in X ) ,
 \end{equation*} 
 where $\Lambda$ denotes the formal adjoint operator of the multiplication by $\omega$. 
 It is known that if $p=n$ and $\Theta(E)$ is semipositive, 
 $\theta$ defines a semipositive Hermitian form. 
 We will use the following norm: 
 \begin{equation*}
   \abs{\alpha}^2_{\theta}
   = \inf 
       \Bigg\{ M \geq 0 \ \Bigg| \begin{matrix} 
                  \ \abs{(\alpha | \beta)}^2 
                  \leq M \cdot \theta(\beta , \beta ) \\ 
                  \ \text{ for any } \ \beta \in (\pigwedge^{n,q} T_{X,x}^*) \otimes E_x\\ 
                                 \end{matrix} \Bigg\}
  \ \in [0,+\infty] 
 \end{equation*}
 for $\alpha \in (\pigwedge^{n,q} T_{X,x}^* ) \otimes E_x$. 

We will obtain Theorem \ref{extension} as a corollary of the following result. 

\begin{thm}\label{L2 extension 1}
 Let $S$ be a $p$-codimensional closed submanifold of a $n$-dimensional projective manifold $X$ with K\"{a}hler form $\omega$. Then there exists a  constant $N=N(S, X,  \omega)>0 $ such that the following holds. 

Fix a positive integer $m\geq1$ , $0 \leq j \leq m-1$ and a holomorphic line bundle $L \to X$ with a smooth Hermitian metric $h$ whose Chern curvature satisfies  
  \begin{equation*}
    \Theta_h \geq N \omega \ \ \ \ \ \text{ on $X$, }
  \end{equation*}
 and $f \in H^0(S,L^{\otimes m})$. 
 Then there exist a constant $C=C(S, X, h)$ and a section $F_j \in H^0(X ,L^{\otimes m})$ such that 
 $J^{j}F_j=f$ and
  \begin{equation*}
         \int_{X} \abs{F_j}_{h^m}^2dV_{\omega, X} 
  \leq C^{(j+1)^2} \int_{S} \abs{f}_{h^m}^2dV_{\omega, S} 
  \end{equation*} 
hold. 
\end{thm}

\begin{proof} The proof is based on an induction on $j$. 
When $j=0$, Theorem \ref{L2 extension 1} is the known result (see \cite{Ohs01}). 
We assume that $F_{j-1}$ was obtained and will construct $F_j$ from $F_{j-1}$ by solving $\bar{\partial}$-equations. The constants $N$ and $C$ will be specified in the induction procedure below. 

Let us first fix a finite system of local coordinates 
\begin{equation*}
 \{s_{\alpha, 1}, \dots, s_{\alpha, n-p}, z_{\alpha, 1}, \dots, z_{\alpha, p}\}_{\alpha}
\end{equation*}
 so that 
  \begin{equation*}
    S \cap U_{\alpha} = \{ z_{\alpha,1}= \cdots =z_{\alpha, p} = 0 \} 
  \end{equation*}
hold. 
 There exists some $\tilde{f} \in C^{\infty}(X, L^{\otimes m})$ such that 
 \begin{equation*}
   \tilde{f}(s, z)= f(s) + O(\abs{z}^m). 
 \end{equation*}
 This is easily seen gluing local $C^{\infty}$-extension by a partition of unity. 
 Fix a smooth cut-off function $\rho: \mathbb{R} \to [0,1]$ satisfying 
 \begin{align*}
   \rho (t) :=  \Bigg\{ \begin{matrix} \ 1 \ \ ( t \leq  \frac{1}{2}) \\
                                       \ 0 \ \ ( t \geq  1 ) \\ 
                    \end{matrix}
        & \ \ \ \ \ \abs{\rho'} \leq 3. 
 \end{align*} 
 Then we set as follows: 
  \begin{align*}
   & G_{\varepsilon}^{(j-1)} := \rho \bigg(\frac{e^{\psi}}{\varepsilon}\bigg)\cdot (\tilde{f}-F_{j-1}) \\
   & g_{\varepsilon} := \bar{\partial}G_{\varepsilon}^{(j-1)} 
       = \underbrace{\bigg(1 + \frac{e^{\psi}}{\varepsilon}\bigg)
                        \rho'\bigg(\frac{e^{\psi}}{\varepsilon}\bigg)
                        \bar{\partial}\psi_{\varepsilon}\wedge (\tilde{f}-F_{j-1})
                        }_{g^{(1)}_{\varepsilon}}
         + \underbrace{\rho \bigg(\frac{e^{\psi}}{\varepsilon}\bigg)\bar{\partial}(\tilde{f}-F_{j-1})
                          }_{g^{(2)}_{\varepsilon}}, 
  \end{align*}
 where 
  \begin{align*}
   & \psi_{\varepsilon} := \log (\varepsilon + e^{\psi})  \ \ \ \ \ 
     \Bigg( \Leftrightarrow  
            1+\frac{e^{\psi}}{\varepsilon} 
           = \frac{e^{\psi_{\varepsilon}}}{\varepsilon}
     \Bigg) \\ 
   & \psi := \log \sum_{\alpha} \chi^2_{\alpha} \sum_{i=1}^{p} \abs{z_{\alpha,i}}^2, 
       \ \ \ \ \ \varepsilon >0. 
  \end{align*}
 Here we choose a smooth function $\chi_{\alpha}$ so that the following hold.  
  \begin{equation*}
    \Supp \chi_{\alpha} \subseteq U_{\alpha}, 
    \ \ \sum_{\alpha} \chi_{\alpha}^2 > 0, 
    \ \ \text{ and } 
    \ \ \sum_{\alpha} \chi_{\alpha}^2 \sum_{i=1}^{p} \abs{z_{\alpha,i}}^2 < e^{-1}  
    \ \ \ \text{ in $X$. }
 \end{equation*} 
 This $\psi$ satisfies the following condition (see \cite{Dem82}, Proposition 1.4). 
 \begin{itemize}
    \setlength{\itemsep}{0pt}
   \item[$(1)$]
     $\psi \in C^{\infty}(X \setminus S)\cap L^1_{\mathrm{loc}}(X)$ \\
     $\psi < -1$ in $X$, $\psi \to -\infty$ around $S$. 
   \item[$(2)$]
     $e^{-p\psi}$ is {\em not} integrable around any point of $S$
   \item[$(3)$]
     There exists a smooth real $(1,1)$-form $\gamma$ in $X$ such that \\
     $\sqrt{-1}\partial\bar{\partial}\psi \geq \gamma$ holds in $X \setminus S$. 
 \end{itemize}
 
 If the equation
  \begin{equation*}
   \begin{cases}
      \bar{\partial}u_{\varepsilon} = \bar{\partial}G_{\varepsilon}^{(j-1)}
      \ \text{ in } \ \Omega \\
      \abs{u_{\varepsilon}}^2 e^{-(j+p)\psi} 
      \ \text{ is locally integrable around } \ S \\ 
   \end{cases}
  \end{equation*}
 has been solved, $u_{\varepsilon} = O(\abs{z}^{(j+1)}) $ along $S$ holds by the above condition
 hence the sequence $ \{G_{\varepsilon}^{(j-1)} - u_{\varepsilon} +F_{j-1}\}_{\varepsilon}$ 
 is expected to converge to desired $F_j$. This is our strategy.   

 To solve $\bar{\partial}$-equations, we quote the following from \cite{Dem00}. 
 \begin{thm}[Ohsawa's modified $L^2$-estimate. \cite{Dem00}, Proposition 3.1]\label{Ohsawa L2 estimate}
  Let $X$ be a complete K\"{a}hler manifold 
  with a K\"{a}hler metric $\omega$ ($\omega$ may not be necessarily complete), 
  $E$ a holomorphic Hermitian line bundle on $X$. 
  Assume that there exist some smooth functions $a,b>0$ and if we set 
   \begin{align*}
    & \Theta'(E) := 
           a \cdot \Theta(E)
             - \sqrt{-1}\partial\bar{\partial}a 
             - \sqrt{-1}b^{-1}\partial a \wedge \bar{\partial}a  \\
    & \theta' ( \alpha , \beta ) := 
               \big( [\Theta'(E),\Lambda]\alpha|\beta \big) 
                 \ \text{ for } 
               \ \alpha, \beta \in (\pigwedge^{n,q}T_{X,x}^*) \otimes E_x 
               \ \ \ \ \ ( x \in X), 
   \end{align*}
  it holds that 
   \begin{equation*}
    \theta' \geq 0 \ \text{ on } \ (\pigwedge^{n,q}T_{X,x}) \otimes E_x 
    \ \ \ \ \ \text{ for any $ x \in X$}. 
   \end{equation*}
Then we have the following. 

 For any $g \in L^2(X, (\pigwedge^{n,q}T_X^*) \otimes E )$ with $\bar{\partial}g=0$ and 
  \begin{equation*}
    \int_X \abs{g}^2_{{\theta}'}dV_{\omega, X} < +\infty, 
  \end{equation*}
 there exists a section $u \in L^2(X, (\pigwedge^{n,q-1}T_X^*) \otimes E)$ 
 with $\bar{\partial}u = g$ such that 
  \begin{equation*}
    \int_X (a+b)^{-1}\abs{u}^2dV_{\omega,X} \leq 2 \int_X \abs{g}^2_{{\theta}'}dV_{\omega,X}.  
  \end{equation*}

\end{thm}

We will apply Theorem \ref{Ohsawa L2 estimate} to $E:=K_X^{-1}\otimes L^{\otimes m}$ and $q=1$. 
 Let us go back to the proof of Theorem \ref{L2 extension 1}. 
 First, we are going to compute 
 \begin{equation}\label{modified Chern curvature}
  \theta_{\varepsilon}' := 
           \big[ a_{\varepsilon}(\Theta(K_X^{-1}\otimes L^{\otimes m}) + (j+p) \sqrt{-1} \partial \bar{\partial} \psi )
                        - \sqrt{-1} \partial\bar{\partial} a_{\varepsilon} 
                        - b_{\varepsilon}^{-1} \sqrt{-1}\partial a_{\varepsilon} \wedge \bar{\partial} a_{\varepsilon}, \Lambda 
           \big]. 
 \end{equation} 
 ($a_{\varepsilon}, b_{\varepsilon}$ will be defined in the following.)
 If we set 
 \begin{equation*}
  a_{\varepsilon} := \chi_{\varepsilon}(\psi_{\varepsilon}) > 0
 \end{equation*}
 for some smooth function $\chi_{\varepsilon}$, it can be computed as: 
 \begin{align*}
  \partial a_{\varepsilon}                        
   & = \chi_{\varepsilon}'(\psi_{\varepsilon}) \partial \psi_{\varepsilon}, \\
  \sqrt{-1} \partial \bar{\partial} a_{\varepsilon} 
   & = \chi_{\varepsilon}'(\psi_{\varepsilon}) \sqrt{-1} \partial \bar{\partial} \psi_{\varepsilon}
      + \chi_{\varepsilon}''(\psi_{\varepsilon}) \sqrt{-1} \partial \psi_{\varepsilon} \wedge \bar{\partial} \psi_{\varepsilon} \\ 
   & = \chi_{\varepsilon}'(\psi_{\varepsilon} ) \sqrt{-1} \partial \bar{\partial} \psi_{\varepsilon}
      + \frac{\chi_{\varepsilon}''(\psi_{\varepsilon})}{\chi_{\varepsilon}'(\psi_{\varepsilon})^2} \sqrt{-1} \partial a_{\varepsilon} \wedge \bar{\partial} a_{\varepsilon} 
 \end{align*} 
 so comparing with (\ref{modified Chern curvature}), it is natural to set 
 \begin{equation*}
  b_{\varepsilon} := 
          - \frac{\chi_{\varepsilon}'(\psi_{\varepsilon})^2}{\chi_{\varepsilon}''(\psi_{\varepsilon})} \ \ (>0). 
 \end{equation*} 
 And we finally define 
 \begin{equation*}
  \chi_{\varepsilon}(t) := \varepsilon -t + \log(1-t). 
 \end{equation*}
 Then for sufficiently small $\varepsilon>0$, we have 
 \begin{align*}
  & a_{\varepsilon} 
    \ \geq \ \varepsilon - \log (\varepsilon + e^{-1}) \ \geq \ 1 \\
  & \sqrt{-1} \partial \bar{\partial} a_{\varepsilon} + b_{\varepsilon}^{-1} \sqrt{-1} \partial a_{\varepsilon} \wedge \bar{\partial} a_{\varepsilon} 
   =\chi_{\varepsilon}'(\psi_{\varepsilon}) \sqrt{-1} \partial \bar{\partial} \psi_{\varepsilon}
    \ \leq \ - \sqrt{-1} \partial \bar{\partial} \psi_{\varepsilon} 
 \end{align*}
 hence 
 \begin{equation*}
  \theta_{\varepsilon}' \ \geq \ \big[ \Theta(K_X^{-1}\otimes L^{\otimes m}) + (j+p)\sqrt{-1} \partial \bar{\partial} \psi + \sqrt{-1} \partial \bar{\partial} \psi_{\varepsilon}, \Lambda \big]. 
 \end{equation*}
 On the other hand, simple computations show:  
 \begin{align*} 
  \partial \psi_{\varepsilon} 
  &= \frac{e^{\psi}}{\varepsilon + e^{\psi}} \partial \psi, \\
  \sqrt{-1}\partial \bar{\partial} \psi_{\varepsilon} 
  &= \frac{e^{\psi}}{\varepsilon + e^{\psi}} \sqrt{-1} \partial \bar{\partial} \psi
   + \frac{e^{\psi}}{\varepsilon + e^{\psi}} \sqrt{-1} \partial \psi \wedge \bar{\partial} \psi 
   - \frac{e^{2\psi}}{(\varepsilon + e^{\psi})^2} \sqrt{-1} \partial \psi \wedge \bar{\partial} \psi \\ 
  &= \frac{e^{\psi}}{\varepsilon + e^{\psi}} \sqrt{-1} \partial \bar{\partial} \psi 
    + \frac{\varepsilon}{e^{\psi}} \sqrt{-1} \partial \psi_{\varepsilon} \wedge \bar{\partial} \psi_{\varepsilon}. 
 \end{align*}
 Therefore, by the compactness of $X$, there exists a constant $N(S, X, \omega)>0$ such that 
 \begin{equation}\label{Nakano positivity}
  \Theta_h \geq N \omega \ \ \ \ \ \text{ on } X
 \end{equation}
 implies 
 \begin{equation}\label{theta is positive}
  \theta_{\varepsilon}' \geq 0 \ \text{ on } \  
    (\pigwedge^{n,1} T_{X,x}^*) \otimes (K_X^{-1}\otimes L^{\otimes m})_x \ \ \ \ \ \text{ for all } x \in X
 \end{equation}
 and eigenvalues of ${\theta}_{\varepsilon}'$ 
 are bounded from below by a positive constant 
 (uniformly with respect to $\varepsilon$) near $S$.

 Next we will estimate $\bar{\partial}\psi_{\varepsilon}$ 
 by $\abs{\cdot}_{{\theta}_{\varepsilon}'}$. 
 Fix arbitrary $\alpha,\beta \in (\pigwedge^{n,1} T_{X,x}^*) \otimes (K_X^{-1}\otimes L^{\otimes m})_x$. 
 By definition, 
 \begin{equation*} 
  \abs{\bar{\partial}\psi_{\varepsilon} \wedge \alpha}^2_{{\theta}_{\varepsilon}'}
  = \inf 
    \Bigg\{ M \geq 0 \ \Bigg| \begin{matrix} 
                               \ \abs{(\bar{\partial}\psi_{\varepsilon}\wedge \alpha | \beta)}^2 
                                 \leq M \cdot \big( [c_{\varepsilon}'(E)\Lambda]\beta|\beta \big) \\ 
                               \ \text{ for any } \ \beta \in (\pigwedge^{n,1} T_{X,x}^*) \otimes E_x\\ 
                                       \end{matrix} \Bigg\} 
 \end{equation*} 
 so it is enough to estimate 
 $\abs{(\bar{\partial}\psi_{\varepsilon}\wedge \alpha | \beta)}^2$. 
 This can be done as follows: 
 \begin{equation*} 
 \begin{split}
  &\abs{(\bar{\partial}\psi_{\varepsilon}\wedge \alpha | \beta)}^2
   = \abs{(\alpha | (\bar{\partial}\psi_{\epsilon})^{\sharp}\beta)}^2 \\
  & \leq \abs{\alpha}^2 \cdot \abs{(\bar{\partial}\psi_{\epsilon})^{\sharp}\beta}^2 
   = \abs{\alpha}^2 \big( ( \bar{\partial}\psi_{\varepsilon})(\bar{\partial}\psi_{\varepsilon})^{\sharp}\beta|\beta \big) 
   = \abs{\alpha}^2 \big( [\sqrt{-1} \partial \psi_{\epsilon} \wedge \bar{\partial} \psi_{\epsilon} , \Lambda ] \beta | \beta \big) 
 \end{split}
 \end{equation*}
 by Shwartz' inequality ($\sharp$ denotes taking the formal adjoint of the multiplication operator), 
 and the last term is bounded by 
 \begin{equation*}
 \begin{split} 
  & \frac{e^{\psi}}{\varepsilon} \abs{\alpha}^2
                    \big( [\sqrt{-1}\partial \bar{\partial} \psi_{\varepsilon}
                       - \frac{e^{\psi}}{\varepsilon+e^{\psi}}\sqrt{-1}\partial \bar{\partial} \psi
                      , \Lambda ] \beta| \beta \big) \\
  & \leq  \frac{e^{\psi}}{\varepsilon} \abs{\alpha}^2
                    \big( [\Theta(K_X^{-1}\otimes L^{\otimes m}) + (j+p)\sqrt{-1} \partial \bar{\partial} \psi 
                           + \sqrt{-1} \partial \bar{\partial} \psi_{\varepsilon}
                           , \Lambda] \beta| \beta \big) \\
  & \leq  \frac{e^{\psi}}{\varepsilon} \abs{\alpha}^2
                    \big( [\Theta'_{\varepsilon}(K_X^{-1}\otimes L^{\otimes m}), \Lambda] \beta| \beta \big). 
 \end{split}
 \end{equation*} 
 Thus we may get a desired estimate  
 \begin{equation}\label{key estimate}
  \abs{\bar{\partial} \psi_{\varepsilon} \wedge \alpha}^2_{{\theta}_{\varepsilon}'}
  \leq \frac{e^{\psi}}{\varepsilon} \abs{\alpha}^2. 
 \end{equation}
 
 This time we estimate 
 $g_{\varepsilon} = g_{\varepsilon}^{(1)} + g_{\varepsilon}^{(2)}$ . 
 By (\ref{key estimate}) and 
 $\Supp g_{\varepsilon}^{(1)} \subseteq \{ e^{\psi} < \varepsilon \}$,  
 $g_{\varepsilon}^{(1)}$ can be estimated. Namely, 
 \begin{equation*}
   \int_{X \setminus S} \abs{g_{\varepsilon}^{(1)}}^2_{{\theta}_{\varepsilon}'}e^{-(j+p)\psi} dV_{\omega, X} 
    \leq 4 \int_{X \setminus S} \abs{\tilde{f}-F_{j-1}}^2_{h^m} \rho'\bigg( \frac{e^{\psi}}{\varepsilon}\bigg)^2e^{-(j+p)\psi} dV_{\omega, X} 
 \end{equation*} 
 holds. Notice $e^{\psi} \sim \sum_{i=1}^{p} \abs{z_{\alpha, i}}^2 $ 
 on $U_{\alpha}$ and $\rho'(t)=0$ for $t<1/2$. By the definition of $F_{j-1}$ we may assume that $U_{\alpha}$ admits a holomorphic local coordinate $z_1, \dots z_p$ in the normal direction along $S$ such that the Taylor expansion of $F_{j-1}$ along $S$ can be written as 
\begin{equation}\label{Taylor}
F_{j-1}(x)=F_{j-1}(s, z) = f(s) + \sum_{k=j}^{\infty}\sum_{i_1+ \dots + i_{p}=k}a_{i_1, \dots i_{p}}(s)z^{i_1}\cdots z^{i_{p}} 
\end{equation}
for any $x=(s, z)\in U_{\alpha}$. The functions $a_{i_1, \dots i_{p}}(s)$ are holomorphic on $s\in S \cap U_{\alpha}$. Therefore by changing variable $z$ to $\varepsilon w$ we get 
 \begin{equation*}
      \limsup_{\varepsilon \to 0} \int_{U_{\alpha} \setminus S} \abs{g_{\varepsilon}^{(1)}}^2_{{\theta}_{\varepsilon}'}e^{-(j+p)\psi} dV_{\omega, X} 
      \leq A_{m} \int_{S \cap U_{\alpha}} \sum_{i_1+\cdots+i_p=j}  \frac{\abs{a_{i_1, \dots, i_p}(s)}^2_{h^m}}{\abs{\bigwedge^p dz}^2} dV_{\omega, S} < +\infty. 
 \end{equation*}
 Here the constant $A_m$ is determined by 
 \begin{equation*}
 \int_{w\in \mathbb{C}^p} {\rho'(\abs{w}^2)}^2 \frac{\bigwedge_{i=1}^p \sqrt{-1}dw_i\wedge d\bar{w_i}}{\abs{w}^{2(j+p)}}. 
 \end{equation*} 
 The point is to estimate the $L^2$-norm of $a_{i_1, \dots, i_p}(s)$. This is done by Cauchy's estimate but here we have to involve the metric $h^m$ so that the constant in the $L^2$-estimate must depends on $h$. The idea here originates from \cite{Pop05} and in our situation we want to see how the coefficients depend on $j$. First we have 
\begin{equation*}
\abs{a_{i_1, \dots, i_p}(s)}^2 \leq  (2\pi)^{-p}R^{-2(j+p)} \int_{\abs{z_i}\leq 2R }\abs{F_{j-1}(s, z)}^2 dz_1\cdots dz_p 
\end{equation*} 
for any small $R>0$. See also subsection $4.2$. 
From this the $L^2$-norm is bounded as follows: 
\begin{equation*}
\int_{S \cap U_{\alpha}}\abs{a_{i_1, \dots, i_p}(s)}^2_{h^m} dV_{\omega, S} \leq (2\pi)^{-p} R^{-2(j+p)} \sup_{\abs{z_i}\leq 2R, s\in U_{\alpha}}\frac{h^m(s, 0)}{h^m(s, z)}\int_X \abs{F_{j-1}}^2_{h^m} d V_{\omega, X}. 
\end{equation*}

Now by the continuity of $h$, There exists a constant $R_0=R_0(X, S, h)$ such that  
\begin{align*}
      \limsup_{\varepsilon \to 0} \int_{U_{\alpha} \setminus S} \abs{g_{\varepsilon}^{(1)}}^2_{{\theta}_{\varepsilon}'}e^{-(j+p)\psi} dV_{\omega, X} 
      \leq  R^{-2(j+p)} \int_X \abs{F_{j-1}}^2_{h^m} d V_{\omega, X} 
\end{align*}
holds for any $R \leq R_0$. Note that $R_0$ is determined by the modulus of continuity of $h$ and independent of $m$ or $j$.  

 We can also estimate $g_{\varepsilon}^{(2)}$.   
 Note that eigenvalues of ${\theta}_{\varepsilon}'$ are bounded below. 
 Then we get  
 \begin{equation*}
   \int_{\Omega \setminus S} \abs{g_{\varepsilon}^{(2)}}^2_{{\theta}_{\varepsilon}'}e^{-(j+p)\psi}dV_{\omega, X} \leq O(\varepsilon) < +\infty
 \end{equation*} 
 because we can see that 
 $\displaystyle \abs{g_{\varepsilon}^{(2)}}^2_{{\theta}_{\varepsilon}'} = O(\abs{z}^m)$ 
 holds in $\Supp g_{\varepsilon}^{(2)} \subseteq \{ e^{\psi} < \varepsilon \} $, 
 by $\bar{\partial}\tilde{f} = O(\abs{z}^{m})$ (using the Taylor expansion).  

 Now we can apply the modified $L^2$-estimate 
 for each $\varepsilon$ in $X \setminus S$. 
 Note that $X \setminus S$ is a complete K\"{a}hler manifold (see \cite{Dem82}, 
 Theorem 1.5). 
 There exists a sequence $ \{ u_{\varepsilon} \} \subseteq L^2(X, L^{\otimes m})$ such that 
 \begin{equation*}
  \int_{X \setminus S} (a_{\varepsilon} + b_{\varepsilon} )^{-1} \abs{u_{\varepsilon}}^2 e^{-(j+p)\psi} dV_{\omega, X} 
  \leq 2 \int_{X \setminus S} \abs{g_{\varepsilon}}^2_{{\theta}_{\varepsilon}'}e^{-(j+p)\psi} dV_{\omega, X} \ < +\infty
 \end{equation*} 
 holds.  

 Let us estimate the left hand side of the inequality. It can be easily seen that 
 \begin{align*}
   & \psi_{\varepsilon} 
     \ \leq \ \log (\varepsilon + e^{-1}) \ \leq -1 \ + O(\varepsilon) \\
   & a_{\varepsilon} 
     \ \leq \ (1+ O(\varepsilon)) \psi^2_{\varepsilon} \\ 
   & b_{\varepsilon} 
     \ = \ (2-\psi_{\varepsilon})^2  
     \ \leq \ (9+ O(\varepsilon)) \psi^2_{\varepsilon} \\
   & a_{\varepsilon} + b_{\varepsilon} 
     \ \leq \ (10 + O(\varepsilon))\psi_{\varepsilon}^2
     \ \leq \ (10 + O(\varepsilon))(-\log(\varepsilon + e^{\psi}))^2 
 \end{align*}
 and 
 \begin{equation*}
  \int_{\Omega} \frac{\abs{G_{\varepsilon}^{(j-1)}}^2}{(\varepsilon+e^{\psi})^{p}(-\log(\varepsilon+e^{\psi}))^2}dV_{\omega, X}
  \leq \frac{M}{(\log \varepsilon)^2} 
 \end{equation*} 
 hold for some constant $M$.  
 Therefore, if we set $F_{\varepsilon} := G_{\varepsilon}^{(j-1)} - u_{\varepsilon}+F_{j-1}$, it follows: 
 \begin{equation*}
 \begin{split}
  & \limsup_{\varepsilon \to 0} \int_{X \setminus S} \frac{\abs{F_{\varepsilon}}^2}{(\varepsilon + e^{\psi})^{p}(-\log(\varepsilon + e^{\psi}))^2}dV_{\omega, X} \\ 
  & \leq  22 \limsup_{\varepsilon \to 0} 
     \int_{X \setminus S} \abs{g_{\varepsilon}}^2_{{\theta}_{\varepsilon}'}e^{-(j+p)\psi}dV_{\omega, X} + \int_X \abs{F_{j-1}}^2_{h^m} dV_{\omega, X} \\
  & \leq (1+R^{-2(j+p)}) \int_X \abs{F_{j-1}}^2_{h^m} dV_{\omega, X}. 
 \end{split}
 \end{equation*}
 By construction, $\bar{\partial}F_{\varepsilon} = 0 $ holds on $X \setminus S$ 
 and in fact also in $X$, thanks to the Riemann extension theorem. 
 
 Finally, Let $\varepsilon \searrow 0$. 
 Then after taking a weakly convergent subsequence, 
 we get a $F_j \in L^2(X, L^{\otimes m})$ 
 such that $\bar{\partial}F_j = 0 $ in $X$ and 
 \begin{equation*}
  \int_X \abs{{F_j}^2}_{h^m}dV_{\omega, X} \leq (1+R^{-2(j+p)}) \int_X \abs{F_{j-1}}^2_{h^m} dV_{\omega, X}. 
 \end{equation*} 
Note that the jet condition $J^{j}F_{\varepsilon}=f$ is preserved under the weak convergence by Lemma \ref{weak L2 convergence implies pointwise convergence} and Cauchy's integral formula. 
By induction we obtain 
 \begin{equation*}
  \int_X \abs{{F_j}^2}_{h^m}dV_{\omega, X} \leq \big(\prod_{k=0}^{j} (1+R^{-2(k+p)})\big) \int_S \abs{f}^2_{h^m} dV_{\omega, S}. 
 \end{equation*} 
\end{proof}

\begin{lem}\label{weak L2 convergence implies pointwise convergence}
  Let $f_k$, $f$ be holomorphic functions defined in a domain 
 $\Omega \subseteq \mathbb{C}^n$. 
  Assume that the sequence $\{f_k\}$  weakly $L^2$-converges to $f$. 
  Then $\{f_k\}$ converges to $f$ pointwise in $\Omega$. 
\end{lem}
\begin{proof}
  Fix any point $x \in \Omega$. 
  Taking $\chi \in C^{\infty}_0(\Omega)$ with $\chi \equiv 1$ near $x$, 
  we have: 
  \begin{equation*} 
    f_k(x) = \int_{\zeta \in \Omega} K^{n,0}_{\mathrm{BM}} (x, \zeta) \wedge \bar{\partial} \chi (\zeta) \wedge f_k(\zeta)
           \to \int_{\zeta \in \Omega} K^{n,0}_{\mathrm{BM}} (x, \zeta) \wedge \bar{\partial} \chi (\zeta) \wedge f(\zeta) = f(x) 
  \end{equation*} 
  by the Koppelman formula. 
  Here $K^{p,q}_{\mathrm{BM}}$ denotes the $(p,q)$-part of the Bochner-Martinelli kernel.    
\end{proof}

\section{Preliminary to section 4}\label{preliminary to section 4}

\subsection{Psh weight}

Let us briefly review the notion of psh weights. Let $L$ be a holomorphic line bundle over a smooth complex projective variety $X$. We usually fix a family of local trivialization patches $U_{\alpha}$ which cover $X$. A singular Hermitian metric $h$ on $L$ is a family of functions $h_{\alpha}=e^{-\varphi_{\alpha}}$ which are defined on corresponding $U_{\alpha}$ and satisfy the transition rule: $\varphi_{\beta}=\varphi_{\alpha}-\log \abs{g_{\alpha \beta}}^2$ on $U_{\alpha}\cap U_{\beta}$. Here $g_{\alpha \beta}$ are transition functions of $L$. {\em The weight functions} $\varphi_{\alpha}$ are assumed to be $L^1$. If $\varphi_{\alpha}$ are smooth, $\{e^{-\varphi_{\alpha}}\}_{\alpha}$ actually defines a smooth Hermitian metric on $L$. We usual denote the family $\{ \varphi_{\alpha} \}_{\alpha}$ as $\varphi$ and omit the index of local trivializations. Notice that each $\varphi=\varphi_{\alpha}$ is only a local function and not globally defined. But the curvature current $\Theta_h=dd^c \varphi$ is globally defined and is semipositive if and only if each $\varphi$ is plurisubharmonic. Here we denote by $d^c$ the real differential operator $\frac{\partial-\bar{\partial}}{4\pi\sqrt{-1}}$.  We call such a weight {\em psh weight} for short. The most important example is the type of weight $\frac{1}{m}\log (\abs{F_1}^2+\cdots +\abs{F_N}^2)$, defined by some holomorphic sections $F_1 \cdots F_N  \in H^0(X, L^{\otimes m})$. Here $\abs{F_i}^2$ ($1\leq i \leq N$) denotes the absolute value taken for the local trivialization of each $F_i$ on $U_{\alpha}$. We call such weights {\em algebraic}. More generally, a psh weight $\varphi$ said to have a {\em small unbounded locus} if the pluripolar set $\varphi^{-1}(-\infty)$ is contained in some closed  proper algebraic subset $S \subseteq X$. 
A singular Hermitian metric $h=e^{-\varphi}$ is said to have strictly positive curvature if $dd^c \varphi \geq \omega$ holds for some K\"{a}hler form $\omega$. 

By the following theorem due to Demailly, one can approximate a psh weight by a  sequence of algebraic weights. 

\begin{prop}\label{Demailly approximation} 
Let $S$ be a smooth projective variety, $L$ a holomorphic line bundle on $S$, $h_S=e^{-\varphi_S}$ a singular Hermitian metric with semipotitive curvature, and $h_0=e^{-\psi}$ a singular Hermitian metric with strictly positive curvature. 
Then there exist $m_0=m_0(S, \psi)\geq 1$ and $C>0$ such that the following holds: Fix $\{f_{m, j}\}_{j=1}^{N_m} \subseteq H^0(S, L^{\otimes m})$, an orthonormal basis with respect to an $L^2$-norm
\begin{equation*}
\int_S \abs{f}^2_{h_S^{(m-m_0)}h_0^{m_0}} dV_{\omega, S} 
 =\int_S \abs{f}^2 e^{-(m-m_0)\varphi_S-m_0\psi} dV_{\omega, S}. 
 \end{equation*}
 Then psh weights $\varphi_{S, m} := \frac{1}{m}\log(\sum_{j=1}^{N_m} \abs{f_{m, j}}^2)$
satisfies 
\begin{align*}
-\frac{\log C }{m} + \bigg( \frac{m-m_0}{m} \varphi_S(s_0) + &\frac{m_0}{m} \psi(s_0) \bigg) \\& \leq \varphi_{S, m}(s_0)  \leq \frac{\log C_r}{m}  + \sup_{s \in B(s_0 ; r)} \bigg( \frac{m-m_0}{m} \varphi_S(s) + \frac{m_0}{m} \psi(s) \bigg) 
\end{align*}
 for any $s_0 \in S$, small $r>0$, and $m \geq m_0$. 
In particular, $\varphi_{S, m}$ converge to $\varphi_S$. 
 \end{prop} 

 Proof of the proposition is on the same line as Proposition $3.1$ in \cite{Dem92}. We omit it here. See also section $12$ of \cite{Dem96} for the related topics. 

\subsection{Positivity of line bundle}

We recall the basic concepts of positivity for line bundles. For a recent account of the theory we refer \cite{Laz04}  in the algebraic course and \cite{Dem96} for the analytic treatments. A line bundle $L$ is said to be ample ({\em resp.} semiample, big) if the associated rational map 
\begin{equation*}
\Phi_m : X \dashrightarrow \mathbb{P}(H^0(X, L^{\otimes m}))
\end{equation*}
is a closed embedding ({\em resp.} a holomorphic map, a birational map to the image) for any sufficiently large $m$. The classical result of Kodaira states that $L$ is ample if and only if it admits a smooth strictly positive curvature Hermitian metric $h_0=e^{-\psi}$. Line bundles satisfy the latter condition usually called {\em positive}. This gives an analytic characterization of ample line bundles and indicates a general principle that positivity of a metric produces holomorphic sections of a line bundle. There also exists an analytic characterization of big line bundle due to Demailly. 
\begin{prop}[Demailly]
A line bundle $L$ is big if and only if it admits a singular Hermitian metric with strictly positive curvature. 
\end{prop}
The proposition allows us to expect some analytic analogue between ample line bundle and big line bundle, and motivates our study for a generalization of a result of \cite{CGZ10} to the big line bundle case.

\section{Extension of singular metric with semipositive curvature}\label{extension of singular metric with semi positive curvature}

In this section, we prove Proposition \ref{metric}. 

\subsection{Construction}

Fix a $p$-dimensional closed submanifold $S\subset X$ and a singular metric $h_S=e^{-\varphi_S}$ on $L|_S$. Moreover, we fix $h_0=e^{-\psi}$, a singular Hermitian metric on $L$ such that $h_0$ has the strictly positive curvature on $X$ and the inequality 
\begin{equation}\label{condition}
\varphi_S \leq \psi|_S
\end{equation} 
holds on $S$.  Such $h_0$ actually exists by the bigness of $L$. 
Denote a positive integer satisfying the condition in Proposition \ref{Demailly approximation} by $m_0$ and the norm of each vector space $H^0(S, L^{\otimes m})$ $(m \geq m_0)$ by 
\begin{equation*}
 \int_S \abs{f}^2 e^{-(m-m_0)\varphi_S-m_0\psi} dV_{\omega, S}. 
\end{equation*}
First of all, by Proposition \ref{Demailly approximation}, we have a sequence of orthonormal basis $\{f_{m, j}\}_{j=1}^{N_m} \subseteq H^0(S, L^{\otimes m})$
such that 
\begin{equation*}
\frac{1}{m}\log \sum_{j=1}^{N_m} \abs{f_{m, j}}^2 \ \to \ \varphi_S \ \ \ \  \text{on} \ \  S. 
\end{equation*}
Note that $\varphi_S$ is not a priori defined on the ambient space $X$. For this reason, we will apply the $L^2$-extension theorem for $\psi$ and then compare the two norm each of which is defined by $\varphi_S$ and $\psi$. At this stage, the assumption $\varphi_S \leq \psi|_S$ will be crucially used. 
At any rate, we have 
\begin{equation*}
\int_S \abs{f_{m, j}}^2 e^{-m\psi} dV_{S, \omega} < +\infty
\end{equation*}
by $\varphi_S \leq \psi|_S$. If a constant $N=N(S, X, \omega)$ (which appears in the assumption in Theorem \ref{extension}) is given, we may assume $dd^c\psi \geq N\omega$ because we can replace $L$ and $\varphi_S$ by $L^{\otimes N}$ and $N\varphi_S$ to prove Proposition \ref{metric}. 
So applying the jet-extension, there exists $F_{m, j} \in H^0(X, L^{\otimes m})$ such that $F_{m, j}|_S= f_{m, j}$ holds and every other term in $(m-1)$-jet along $S$ vanishes.  That is, for any point $s_0 \in S$ there exists a neighborhood $U$ on $X$ and a holomorphic local coordinate $z_1, \dots z_p$ in the normal direction along $S$, which centers at $s_0$, such that the Taylor expansion of $F_{m, j}$ along $S$ can be written as 
\begin{equation}\label{Taylor}
F_{m, j}(x)=F_{m, j}(s, z) = f_{m, j}(s) + \sum_{k=m}^{\infty}\sum_{i_1+ \dots + i_{p}=k}a_{i_1, \dots i_{p}}(s)z^{i_1}\cdots z^{i_{p}} 
\end{equation}
for any $x=(s, z)\in U$. 
Note that the functions $a_{i_1, \dots i_{p}}(s)$ are holomorphic on $s\in U \cap S$. 
Moreover, if the coefficient of $L^2$-estimate is refined, 
\begin{equation}\label{L2}
\int_X \abs{F_{m, j}}^2 e^{-m\psi} dV_{X, \omega} \leq C_1^m \int_S \abs{f_{m, j}} e^{-m\psi} dV_{S, \omega}
\end{equation}
hold for each $m$ and $j$. Here the constant $C_1=C_1(S, X, \psi)$ does not depend on $m$.  
Let us define 
\begin{equation}\label{definition}
\varphi_m:=\frac{1}{m}\log \sum_{j=1}^{N_m} \abs{F_{m, j}}^2. 
\end{equation}
This $\varphi_m$ actually defines an algebraic singular metric on $L$ over $X$. 
As a consequence of the above $L^2$-estimates one can derive an upper boundedness of $\varphi_m$. Actually by the mean value property of plurisubharmonic function one has 
\begin{align*}
\abs{F_{m, j}(x)}^2 \leq \frac{n!}{\pi^n r^{2n}}\int_{B(x, r)} \abs{F_{m, j}}^2 \ \ \leq \frac{n!}{\pi^n r^{2n}} \sup_{B(x, r)}e^{m\psi} \int_X\abs{F_{m, j}}^2e^{-m\psi}  dV_{X, \omega}
\end{align*}
for any small ball $B(x, r) \subset X$. Then (\ref{L2}) yields 
\begin{equation*}
\abs{F_{m, j}(x)}^2 \leq {C_2}^m \int_S \abs{f_{m, j}}^2 e^{-m\psi} dV_{S, \omega}. 
\end{equation*}
The constant $C_2$ here depends not only $S$ but also $\psi$. Let us exchange $\psi$ for $\varphi_S$ by the inequality 
\begin{equation*}
\int_S \abs{f_{m, j}}^2 e^{-m\psi} dV_{S, \omega} \leq \sup_S e^{(m-m_0)(\varphi_S-\psi)} \int_S \abs{f_{m, j}}^2 e^{-(m-m_0)\varphi_S-m_0\psi} dV_{S, \omega}. 
\end{equation*}
Note that $\varphi_S-\psi$ is well-defined as a global function on $S$ and the integral on the right-hand side equals to $1$ by the normality. So the assumption $\varphi_S \leq \psi|_S$ shows that the right-hand side is not greater than $1$. Summarizing up, we have 
\begin{equation}\label{C_2}
\abs{F_{m, j}(x)}^2 \leq {C_2}^m. 
\end{equation}
Then, by the fact $N_m = O(m^{n-p})$ derived from the Riemann-Roch formula, there exists a constant $C_3$ only depends on $S$ and $\psi$ such that 
\begin{equation}\label{boundedness}
\varphi_m \leq C_3 
\end{equation}
holds. 
This constant $C_3$ is independent of $m$ and therefore we get a subsequence of $\varphi_m$ , which converges to a psh weight $\varphi$. From a general theory of plurisubharmonic function, one may assume 
\begin{equation*}
\varphi(x) = {\limsup_{m \to \infty}}^* \varphi_m(x) := \limsup_{y \to x} \limsup_{m \to \infty} \varphi_m(y). 
\end{equation*}
Here the symbol $*$ denotes the upper-semicontinuous envelop. 
This ends the construction of $\varphi$. By the definition, $\varphi|_S\geq \varphi_S$ is clear. Our remained task is to show the converse inequality. 

\subsection{Upper bound estimate}

At this time we need the control of $(m-1)$-jet of each $F_{m, j}$. For a general convergence of plurisubharmonic functions, a value of the upper-semicontinuous envelop happens to jump at some point. But we may expect that if every $\sum_j F_{m, j}$ sufficiently tangents to $S$, this is not the case. Let us check it in our situation. We fix $s_0\in S$ and its neighborhood $U$ to have (\ref{Taylor}). Note that we may concentrate on each point's neighborhood to get requiring upper bound estimate of $\varphi$, thanks to the compactness of $S$. Take $R>0$ so that for any $1\leq i \leq p$, $\abs{z_i} <3R$ holds on $U$. 

We first estimate $a_{i_1, \dots, i_p}$in (\ref{Taylor}). This is done by Cauchy's integral formula: 
\begin{equation*}
a_{i_1, \dots, i_p}(s) = \bigg(\frac{1}{2\pi \sqrt{-1}}\bigg)^p \int_{\abs{z_p}=r} \cdots \int_{\abs{z_1}=r} \frac{F_{m, j}(s, z)}{z_1^{i_1+1}\cdots z_p^{i_p+1}}dz_1\cdots dz_p. 
\end{equation*}
In fact by Cauchy-Schwarz' inequality one can extract the $L^2$-norm as: 
\begin{align*}
\abs{a_{i_1, \dots, i_p}(s)}^2 
& \leq \bigg( \frac{1}{2\pi}\bigg)^{2p}\bigg( \int_0^{2\pi} \cdots \int_0^{2\pi} \frac{\abs{F_{m, j}(s, z)}}{r^{i_1+\cdots i_p}} d\theta_1 \cdots d\theta_p\bigg)^2 \\
& \leq \bigg( \frac{1}{2\pi}\bigg)^{2p}\bigg( \int_0^{2\pi} \cdots \int_0^{2\pi} \abs{F_{m, j}(s, z)}^2 d\theta_1 \cdots d\theta_p \bigg) \bigg(  \int_0^{2\pi} \cdots \int_0^{2\pi}\frac{d\theta_1 \cdots d\theta_p}{r^{2(i_1+\cdots i_p)}} \bigg) \\
& = \bigg( \frac{1}{2\pi}\bigg)^{p} \frac{1}{r^{2(i_1+\cdots i_p)}} \int_0^{2\pi} \cdots \int_0^{2\pi} \abs{F_{m, j}(s, z)}^2 d\theta_1 \cdots d\theta_p. 
\end{align*}
Integrating the both side from $r_i=R$ to $r_i=2R$ ($1\leq i \leq p$), we get 
\begin{align*}
&\int_R^{2R} \cdots \int_R^{2R} \abs{a_{i_1, \dots, i_p}(s)}^2 dr_1 \cdots dr_p \\
&\ \ \ \ \ \ \ \ \ \ \ \ \leq \bigg( \frac{1}{2\pi} \bigg)^p R^{-2(i_1+\cdots i_p)-p} \int_{R \leq \abs{z_i} \leq 2R} \abs{F_{m, j}(s, z)}^2 dz_1\cdots dz_p. 
\end{align*}  
Thus it holds that 
\begin{equation*}
\abs{a_{i_1, \dots, i_p}(s)} \leq  \bigg( \frac{1}{2\pi} \bigg)^{\frac{p}{2}} \bigg( \int_{\abs{z_i}\leq 2R }\abs{F_{m, j}(s, z)}^2 dz_1\cdots dz_p \bigg)^{\frac{1}{2}}R^{-(i_1+\cdots i_p)-p}.  
\end{equation*}

Take $0<\varepsilon<1/2$.  Then in the ball $\{ \abs{z} \leq p^{-1} \varepsilon R \} \subset \mathbb{C}^p$ (here we fix $s$ and move only $z$) we have the following estimates: 
\begin{align*}
&\abs{F_{m, j}(s, z) -f_{m, j}(s)} 
\leq \sum_{i_1+\cdots i_p \geq m} \abs{a_{i_1, \dots i_p}(s)}\cdot \abs{z}^{i_1+\cdots +i_p} \\
&\leq \bigg( \frac{1}{2\pi} \bigg)^{\frac{p}{2}}\bigg( \int_{\abs{z_i}\leq 2R }\abs{F_{m, j}(s, z)}^2 dz_1\cdots dz_p\bigg)^{\frac{1}{2}} \sum_{i_1+\cdots i_p \geq m} R^{-(i_1+\cdots i_p)-p}  \Big(\frac{\varepsilon R}{p}\Big)^{i_1+\cdots i_p} \\
&\leq \bigg( \frac{1}{2\pi} \bigg)^{\frac{p}{2}}\bigg( \int_{\abs{z_i}\leq 2R }\abs{F_{m, j}(s, z)}^2 dz_1\cdots dz_p\bigg)^{\frac{1}{2}} 2R^{-p} {\varepsilon}^{m}. 
\end{align*}
From (\ref{C_2}), the integral in the last term is bounded by a constant $C_2^m$. Thus we obtain: 
\begin{equation}\label{jet}
\abs{F_{m, j}(s, z)-f_{m, j}(s)} \leq C_2^m \varepsilon^m. 
\end{equation}
Here the constant $C_2$ depends only on $S$ and $\psi$. This is what we need. 

Now let us estimate $\varphi_m$. By the triangle inequality 

\begin{align*}
\varphi_m(s, z)
&= \frac{2}{m}\log \Big(  \sum_{j=1}^{N_m} \abs{F_{m, j}(s, z)}^2 \Big)^{\frac{1}{2}} \\
&\leq \frac{2}{m}\log \bigg[ \Big(\sum_{j=1}^{N_m} \abs{f_{m, j}(s)}^2 \Big)^{\frac{1}{2}} + \Big( \sum_{j=1}^{N_m} \abs{F_{m, j}(s, z)-f_{m, j}(s)}^2 \Big)^{\frac{1}{2}} \bigg]
\end{align*}
hold. Thanks to (\ref{jet}), the second term in the logarithm is bounded from above by the square root of $N_mC_2^m\varepsilon^m$. 
The first term can be bounded from above by Proposition \ref{Demailly approximation}. Thus for any $s\in B(s_0 ; r)$ we have 
\begin{align*}
 \varphi_m(s, z)  \leq \frac{2}{m} \log \bigg[ \big( \sup_{s' \in B(s_0 ; 2r)} C_re^{(m-m_0)\varphi_S(s') + m_0\psi(s')} \big)^{\frac{1}{2}} +(N_mC_2^m\varepsilon^m)^{\frac{1}{2}}\bigg]. 
\end{align*}
By the concavity of the logarithmic function 
\begin{equation*}
\log(a+b) \leq \log a + \frac{b}{a}
\end{equation*}
holds for any $a, b>0$ so it yields 
\begin{align*}
 \varphi_m(s, z) \leq \frac{1}{m}\log C_r &+ \sup_{s' \in B(s_0 ; 2r)} \bigg(\frac{(m-m_0)}{m} \varphi_S(s') + \frac{m_0}{m}\psi(s') \bigg) \\ 
 &+\frac{2}{m}\frac{(N_mC_2^m\varepsilon^m)^{\frac{1}{2}}}{\big(\sup_{s' \in B(s_0 ; 2r)} e^{(m-m_0) \varphi_S(s') +m_0\psi(s')}\big)^\frac{1}{2}}. 
\end{align*}
Note that the denominator of the third term on the right-hand side 
\begin{equation*}
\big(\sup_{s' \in B(s_0 ; 2r)} e^{(m-m_0) \varphi_S(s') +m_0\psi(s')}\big)^\frac{1}{2}
\end{equation*}
is positive but may go to $0$ when $r\to0$. 
However, taking $\varepsilon$ sufficiently small for $r$, we can make the third term smaller than $1/m$. 
Hence letting $m \to \infty$ we get 
\begin{equation*}
\limsup_{(s, z) \to 0} \limsup_{m \to \infty} \varphi_m(s, z) \leq \varphi_S(s_0). 
\end{equation*}
This completes the proof of Proposition \ref{metric}. 

\subsection{Remarks on restricted volumes}\label{Remarks on restricted volumes}

Before concluding the paper, we discuss the assumption which is necessary for generalizing Theorem \ref{CGZ} to big line bundles. For the proof of Proposition \ref{metric}, especially on the ample line bundle case, one can skip this subsection. First we briefly review an analytic representation of the volume of a line bundle along a subvariety. 

\begin{dfn}
Let $S$ be a $p$-codimensional subvariety on a $n$-dimensional smooth projective variety $X$. Denote by $H^0(X|S, L^{\otimes m})$ the image of the restriction map $H^0(X, L^{\otimes m}) \to H^0(S, L^{\otimes m})$. The restricted volume of a line bundle $L$ along $S$ is defined to be 
\begin{equation*}
\vol_{X|S}(L):=\limsup_{m\to\infty}\frac{\dim H^0(X|S, L^{\otimes m})}{m^{n-p}/(n-p)!}. 
\end{equation*}
When $S=X$, we simply write it as $\vol_{X}(L)$. 
\end{dfn}
It is easy to show that if $L$ is ample (more generally semiample), $\vol_{X|S}(L)=\vol_{S}(L|_S)=(S.L^{n-p})$ holds, {\em i.e.} the restricted volume equals to the intersection number. In general restricted volumes are invariant under birational transformations and a line bundle $L$ is big if and only if $\vol_{X}(L)>0$. In some sense they measure the positivity of line bundles along subvarieties. An analytic description of volumes was first given in \cite{His11}. For a singular Hermitian metric $h=e^{-\varphi}$ on $X$, let us denote by $\langle (dd^c\varphi|_S)^{n-p} \rangle$ the non-pluripolar Monge-Amp\`{e}re product of $\varphi|_S$, a positive measure which coincides with usual Monge-Amp\`{e}re product $(dd^c\varphi|_S)^{n-p}$ when $\varphi$ is smooth. For the precise definition, see \cite{BEGZ10}. 
\begin{thm}[Theorem $1.3$ in \cite{His11}]\label{volume}
Assume there exists a singular Hermitian metric $h_0=e^{-\psi}$ with strictly positive curvature on $L$, and $S \subsetneq \psi^{-1}(-\infty)$ hold. Then it holds that 
\begin{equation}
 \vol_{X|S}(L)= \max_{\varphi} \int_S \langle (dd^c\varphi|_S)^{n-p} \rangle, 
\end{equation}
where $e^{-\varphi}$ runs through all the singular Hermitian metric on $L$ over $X$. 
\end{thm}

We can relate the analytic study of volumes to the problem of metric extension. 
\begin{cor}
In general $\vol_{X|S}(L) \leq \vol_S(L|_S)$ and if every semipositive curvature singular metric on $L|_S$ is extendable $\vol_{X|S}(L) = \vol_S(L|_S)$ holds. 
\end{cor}

\begin{exam}
Let $\mu: X \to \mathbb{P}^2$ be the one point blow up and $E$ an exceptional divisor. We take $L=\mu^*\mathcal{O}(1)\otimes \mathcal{O}(E)$ and $S$ to be the strict transform of a line which through the blown up point. Then it is easy to see that this example satisfies the assumption of Theorem \ref{volume} but $\vol_{X|S}(L)=1$ and $\vol_S(L|_S)=2$. 
\end{exam}
The next example shows that there exists a non-extendable metric even if we assume $L$ is semiample and big. 
\begin{exam}\label{semiample}
Let $\mu: X \to \mathbb{P}^2$ be the one point blow up and $E$ an exceptional divisor as the above. Take $L=\mu^*\mathcal{O}(1)$ and $S$ to be the smooth strict transform of a nodal curve. Then $L$ has a singular metric with semipositive curvature, which can not be extended to $X$. In fact one can construct a semipositive curvature metric $h_S=e^{-\varphi_S}$ weight on the ample line bundle $L|_S$, which is $+\infty$ on a point in $E \cap S$ and smooth on another. If a psh  extension $\varphi_S$ exists, it is constant along $E$ since $L|_E$ is trivial. This contradicts the definition of $h_S$. 
\end{exam}

These examples naturally indicate to us the assumption in Problem\ref{problem}. \\


{\bf Acknowledgments.}
The author would like to express his gratitude 
to his advisor Professor Shigeharu Takayama for his warm encouragements, 
many helpful suggestions and reading the drafts. 
The author specially wishes to thank Professor Dan Coman for his pointing out a critical mistake in the first version of the preprint. 
The author wishes to thank Professor Vincent Guedj and Professor Ahmed Zeriahi for several helpful comments concerning the approach of $L^2$-extension, during the conference {\em Complex and Riemannian Geometry} held at CIRM, Luminy. 
The author also would like to thank Shin-ichi Matsumura 
for several helpful comments, especially on Example \ref{semiample}. 
This research is supported by JSPS Research Fellowships for Young Scientists (22-6742). 
 

\end{document}